\documentclass[12pt]{amsart}
  \usepackage{latexsym} 
  \usepackage[all]{xy}
  \usepackage{amsthm} 
    \usepackage{amsfonts} 
   \usepackage{amsmath} 
  \usepackage{amssymb}

  \def\cC{{\mathcal C}} 
  \def\cO{{\mathcal O}}
 \def\CC{{\mathbb C}}
 \def\NN{{\mathbb N}}
 \def\PP{{\mathbb P}}
\def\RR{{\mathbb R}}
\def\ZZ{{\mathbb Z}}

\def\bdelta{\bar{\delta}}

\def\ot{\otimes}

 \newtheorem{proposition}{Proposition}[section]
  \newtheorem{lemma}[proposition]{Lemma} 
   
  \newtheorem{theorem}[proposition]{Theorem} 

\theoremstyle{definition}

  \theoremstyle{remark}

\begin{document}

\title{Complex geometry of quantum cones}

\author{Tomasz Brzezi\'nski}
 \address{ Department of Mathematics, Swansea University, 
  Swansea SA2 8PP, U.K.} 
  \email{T.Brzezinski@swansea.ac.uk}   

\begin{abstract}
The algebras obtained as fixed points of the action of the cyclic group $\ZZ_N$ on the coordinate algebra of the quantum disc are studied. These can be understood as coordinate algebras of quantum or non-commutative cones. The following observations are made. First, contrary to the classical situation, the actions of $\ZZ_N$ are free and the resulting algebras are homologically smooth. Second, the quantum  cone algebras  admit differential calculi that have all the characteristics of calculi on {\em smooth} complex curves. Third, the corresponding volume forms are exact, indicating that the constructed algebras describe manifolds with boundaries.
\end{abstract}

\maketitle                   

\section{Introduction}
Recent studies of  deformed orbifolds or manifolds with singularities revealed  an intriguing phenomenon: many classical surfaces $\Sigma$ with singular points admit non-commutative or $q$-versions $\Sigma_q$ that are smooth. Smoothness can be understood on several levels. Most naively, one observes that the multiple roots of the defining polynomial relations become separated upon the quantization  \cite{BrzFai:tea}, \cite{Brz:Sei}. Furthermore, the classically non-free action of a finite group on a smooth surface that gives rise to the singularity becomes free once the surface and the action are quantized  \cite{Brz:smo}. On a more sophisticated level, the deformed coordinate algebras $\cO(\Sigma_q)$ are {\em homologically smooth}, i.e.\ they fit into a finite exact sequence of finitely generated and projective $\cO(\Sigma_q)$-bimodules \cite{Brz:smo}. Often, $\cO(\Sigma_q)$ is  a (twisted) {\em Calabi-Yau algebra} of dimension two, i.e\  $\cO(\Sigma_q)$ is not only homologically smooth, but also its Hochschild cohomology with values in the enveloping algebra of $\cO(\Sigma_q)$ is trivial except in degree two, where it is isomorphic to $\cO(\Sigma_q)$ with the bimodule structure twisted by an automorphism \cite{Liu:hom}. This last observation can be interpreted as an existence of the volume form, and is tantamount  to the Poincar\'e duality \cite{Van:rel}. 

Every two-dimensional orbifold is obtained by glueing cones, hence quantum cones seem to be a natural point to start. In this note we carry out parts of the programme of {\em differential smoothing} initiated in \cite{BrzSit:pil} and discuss differential structures on quantum cones that allow one to view them as complex curves.   We define coordinate algebras of quantum cones $\cO(C^N_{q,\gamma})$  as fixed points of the action of cyclic groups $\ZZ_N$ on the non-commutative disc algebra $\cO(D_{q,\gamma})$ \cite{KliLes:two}. This action turns out to be free and  $\cO(C^N_{q,\gamma})$  are (twisted) Calabi-Yau algebras. We derive differential calculi on $\cO(C^N_{q,\gamma})$ from that on $\cO(D_{q,\gamma})$ and show that they define holomorphic structures on $\cO(C^N_{q,\gamma})$ in the sense of \cite{KhaLan:hol}. The resulting volume forms on $\cO(C^N_{q,\gamma})$ are exact, which might indicate that quantum cones have boundaries.

\section{Quantum cones  are smooth}\label{sec.smooth}
The coordinate algebra $\cO(D_{q,\gamma})$  of the quantum disc is defined as the complex $*$-algebra generated by $z,z^*$ subject to relation
\begin{equation}\label{disc}
z^*z - qzz^* = \gamma,
\end{equation}
where $q,\gamma \in \RR$, $q\neq 0$ are parameters. This is obviously a deformation of the plane (or the unit disc), which corresponds to $q,\gamma =0$.  $\cO(D_{q,0})$ is the quantum or Manin's plane, and the case $\cO(D_{0,1})$ is the quantum oscillator algebra. In fact, if $\gamma\neq 0$, then the generators can be rescaled, so that $\gamma$ can be eliminated from \eqref{disc}. If $q\in (0,1)$,  the algebra   $\cO(D_{q,1-q})$ can be completed to the $C^*$-algebra isomorphic to the  Toeplitz  algebra \cite{KliLes:two}. For an accessible introduction to geometry of the quantum disc we refer to \cite[Chapter~1]{Vak:qua}.

$\cO(D_{q,\gamma})$  admits an action of the cyclic group $\ZZ_N$, which can be described in terms of the $\ZZ_N$-grading on generators as $\deg(z) =1$, $\deg(z^*) = N-1 = -1\  {\rm mod}\ N$. The fixed point subalgebra of this action or the degree-zero part of  $\cO(D_{q,\gamma})$ is generated by $a=zz^*$, $b=z^N$, $b^* = z^{*N}$. These satisfy the  relations
\begin{equation}\label{cone}
ab = q^N ba + \gamma [N]_q b, \qquad bb^* = \prod_{l=0}^{N-1} (q^{-l}a + \gamma [-l]_q), \qquad  b^*b = \prod_{l=1}^{N} (q^{l}a + \gamma [l]_q),
\end{equation}
where $[n]_q = (1-q^{n})/(1-q)$ are $q$-integers. The complex $*$-algebra generated by a selfadjoint  $a$ and $b,b^*$ subject to relations \eqref{cone} is called the {\em coordinate algebra of the quantum $N$-cone} and is denoted by $\cO(C^N_{q,\gamma})$. In case $q=\gamma =0$, $a$, $b$, $b^*$ can be thought of as coordinates and then relations \eqref{cone} describe a cone in $\CC\times \RR$, hence the name. The polynomials in $a$ on the right-hand sides of the last two equations in \eqref{cone} have no repeated roots as long as $\gamma\neq 0$. One might therefore expect that in this case the algebras  $\cO(C^N_{q,\gamma})$  describe {\em smooth} non-commutative surfaces. 

An action of a (finite) group $G$   on a quantum space $\Sigma_q$ is free if and only if the corresponding coordinate algebra $\cO(\Sigma_q)$ is {\em strongly $G$-graded}. This means that for all $g\in G$, one can find finite number of elements $X_i $ of $G$-degree $g$ and $Y_i$ of degree $g^{-1}$ such that $\sum_i X_iY_i =1$. In the case of the cyclic group $\ZZ_N$ suffices it to find such $X_i$ of degree $1$ and $Y_i$ of degree $N-1$; see \cite[Section~AI.3.2]{NasVan:gra}.

\begin{theorem}\label{thm.free}
For all $\gamma\neq 0$, the $\ZZ_N$-action on $D_{q,\gamma}$ is free. Consequently, $\cO(C^N_{q,\gamma})$ are homologically smooth algebras for all values of $N>1$, $q$ and $\gamma \neq 0$. 
\end{theorem}
\begin{proof}
Set $X_0 = \alpha_0 z^{*N-1}$, $X_i = \alpha_i a^{i-1} z$, $Y_0 =  z^{N-1}$ and $Y_i = z^*$, $i=1,\ldots, N-1$, where all the $\alpha_r\in \CC$ are to be determined. By equations \eqref{disc}, 
$$
z^{*N-1} z^{N-1} = \prod_{l=1}^{N-1} (q^{l}a + \gamma [l]_q)  =: \sum_{r=0}^{N-1} \beta_r a^r.
$$
 The comparison of the coefficients of the powers of $a$ in the condition $\sum_{i=0}^{N-1} X_iY_i =1$, yields the system of $N$ equations for the $\alpha_r$: $\alpha_0 \beta_0 = 1$ and $\alpha_0 \beta_i +\alpha _i =0$, for $i=1,2,\ldots, N-1$, whose
determinant is $\beta_0 = \gamma^{N-1} \prod_{l=1}^{N-1}[l]_q\neq 0$. Hence the required $X_i$ and $Y_i$ can be found.

Since the enveloping algebra of $\cO(D_{q,\gamma})$ is left Noetherian of finite global dimension and since the $\ZZ_N$-actions defining the cones are free if $\gamma\neq 0$, the algebras $\cO(C^N_{q,\gamma})$ are homologically smooth by \cite[Corollary~6]{Kra:Hoc} or \cite[Criterion~1]{Brz:smo}.
\end{proof}

The freeness of the action of $\ZZ_N$ on $D_{q,\gamma}$ is equivalent to the statement that the quantum disc is the total space of a $\ZZ_N$-principal quantum bundle over the quantum cone $C^N_{q,\gamma}$ \cite{BrzMaj:gau}, \cite{BrzHaj:Che}. This bundle is non-trivial, i.e.\ $\cO(D_{q,\gamma})$ is not isomorphic to $\cO(C^N_{q,\gamma})\ot \CC\ZZ_N$ as a $\ZZ_N$-graded left $\cO(C^N_{q,\gamma})$-module. The triviality would require the existence of  invertible elements in $\cO(D_{q,\gamma})$ of degree $n$, for  all $n\in \ZZ_N$. However, only non-zero scalar multiples of the identity are invertible in $\cO(D_{q,\gamma})$, and these have degree zero. 

\begin{theorem}\label{thm.C-Y}
For all $\gamma\neq 0$,  $\cO(C^N_{q,\gamma})$ are (twisted) Calabi-Yau algebras. 
\end{theorem}
\begin{proof}
Define $x:= 1 + \gamma^{-1}(q-1) a$. The relations \eqref{cone} yield
\begin{equation}\label{cone.z}
xb = q^Nbx, \qquad bb^* = \left(\frac{\gamma}{1-q}\right)^N\prod_{l=0}^{N-1} \left(1-q^{-l}x\right), \qquad  b^*b = \left(\frac{\gamma}{1-q}\right)^N\prod_{l=1}^{N} \left(1-q^{l}x\right).
\end{equation}
Since the change of variables $a\mapsto x$ is reversible, equations \eqref{cone.z} simply give a different presentation of $\cO(C^N_{q,\gamma})$. This presentation allows one to view $\cO(C^N_{q,\gamma})$ as a {\em generalized Weyl algebra} over the polynomial ring in the variable $x$ \cite{Bav:fin}. Since the polynomials on the right hand side of \eqref{cone.z} have no repeated roots, $\cO(C^N_{q,\gamma})$ are twisted Calabi-Yau algebras by \cite[Theorem~4.5]{Liu:hom}.
\end{proof}

The inspection of relations \eqref{cone.z} makes it clear that, for $\gamma\neq 0$ and $N$ odd, the algebras $\cO(C^N_{q,\gamma})$ are isomorphic to coordinate algebras of quantum real weighted projective spaces $\cO(\RR\PP^2_q(N;+))$ \cite{Brz:Sei}.

\section{Differential and complex geometry of quantum cones}\label{sec.complex}
By a {\em differential calculus} over an algebra $A$ we mean a pair $(\Omega (A), d)$, where $\Omega (A) = \oplus_{n\in \NN}  \Omega^n (A)$ is an $\NN$-graded algebra  such that $\Omega^0 (A) = A$, and  $d: \Omega^ n(A) \to \Omega^{n+1}(A)$  is the degree-one linear map that satisfies the graded Leibniz rule, is such that $d\circ d=0$ and, for all $n\in \NN$,  $\Omega^n(A)=Ad(A)d(A)\cdots d(A)$ ($d(A)$ appears $n$-times).  If $A$ is a complex $*$-algebra, then it is often requested that $\Omega (A)$ be a $*$-algebra and that $*\circ d = d\circ *$. In this situation one refers to $\Omega (A)$ as to a {\em $*$-differential calculus}. In this case, following \cite{KhaLan:hol}, a {\em complex structure on $A$}  is  the bi-grading decomposition of $\Omega (A)$,
$$
\Omega^n(A) = \bigoplus_{k+l =n} \Omega^{(k,l)}(A),
$$
such that $*:\Omega^{(k,l)}(A)\to \Omega^{(l,k)}(A)$, and the decomposition $d= \delta +\bdelta$ into differentials $\delta: \Omega^{(k,l)}(A)\to \Omega^{(k+1,l)}(A)$, $\bdelta: \Omega^{(k,l)}(A)\to \Omega^{(k,l+1)}(A)$, such that 
$\delta(a)^* = \bdelta(a^*)$, for all $a\in \Omega (A)$.

To construct a complex structure on $\cO(C^N_{q,\gamma})$ 
start with the  differential $*$-calculus $\Omega (D_{q,\gamma})$ on $\cO(D_{q,\gamma})$ generated by one-forms $dz$ and $dz^*$ subject to relations
\begin{equation}\label{calculus.d}
zdz = q^{-1}dz z, \qquad z^*dz = qdz z^*, \qquad dz\wedge dz^* = -q^{-1}dz^*\wedge dz, \qquad dz\wedge dz =0;
\end{equation}
see e.g.\ \cite{SinVak:ana}. Note that $\Omega (D_{q,\gamma}) = \cO(D_{q,\gamma}) \oplus \Omega^1 (D_{q,\gamma}) \oplus \Omega^2 (D_{q,\gamma})$, where $\Omega^1 (D_{q,\gamma})$ is a a free left and right  $\cO(D_{q,\gamma})$-module with free generators $dz$ and $dz^*$, and   $\Omega^2 (D_{q,\gamma})$ is a free  $\cO(D_{q,\gamma})$-module generated by $dz\wedge dz^*$. We define the differential $*$-calculus $\Omega (C^N_{q,\gamma})$ on $\cO(C^N_{q,\gamma})$ as the differential $*$-subalgebra of $\Omega (D_{q,\gamma})$  generated by $a=zz^*$, $b= z^N$ and their differentials $da$, $db$. Set
$$
\omega_0 := dz z^*, \qquad \omega_1 := d b = [N]_qz^{N-1} dz, \qquad \omega:= dz\wedge dz^*.
$$

\begin{theorem}\label{thm.complex}
Assume that $\gamma\neq 0$. Let $\Omega^{(1,0)} (C^N_{q,\gamma})$ be the submodule of $\Omega^{1} (C^N_{q,\gamma})$ generated by $\omega_0$, $\omega_1$, let $\Omega^{(0,1)} (C^N_{q,\gamma})$ be the submodule of $\Omega^{1} (C^N_{q,\gamma})$ generated by $\omega_0^*$, $\omega_1^*$, and let $\Omega^{(1,1)} (C^N_{q,\gamma}) = \Omega^{2} (C^N_{q,\gamma})$. Set $\Omega^{(k,l)} (C^N_{q,\gamma}) =0$ if $k>1$ or $l>1$. Let $\delta$ be the projection of $d$ onto $\Omega^{(1,0)} (C^N_{q,\gamma})$ and let $\bdelta$ be the projection of $d$ onto $\Omega^{(0,1)} (C^N_{q,\gamma})$. Then the above bi-grading decomposition of $\Omega (C^N_{q,\gamma})$ defines a complex structure on $\cO(C^N_{q,\gamma})$. Furthermore, $\Omega^{2} (C^N_{q,\gamma}) = \cO(C^N_{q,\gamma})\, \omega = \omega\, \cO(C^N_{q,\gamma})$.
\end{theorem}
\begin{proof}
The only non-trivial statement to be proven here is that both $\omega_0$ and $\omega_0^*$ belong to $\Omega^{1} (C^N_{q,\gamma})$ (there sum is in $\Omega^{1} (C^N_{q,\gamma})$, since it is equal to $da$). If this is established, then the bi-graded decomposition of $\Omega^{1} (C^N_{q,\gamma})$ and maps $\delta$, $\bdelta$ are well-defined. Also, since $\omega = -d\omega_0$, the last assertion will follow. We start by proving the following
\begin{lemma}\label{lem.calculus}
For all $q\in (0,1)$ and positive $n\in \NN$,  if an ideal in the polynomial ring $\CC[w]$ contains both $X_n(w)= \prod_{l=1}^{n-1} (1+q^{-l}w)$ and $ Y_n(w)= \prod_{l=1}^{n-1} (1+q^{l}w)$, then it also contains the identity. 
\end{lemma}
\begin{proof}
By induction in $n$. Clearly, the statement is true for $n=1$. Assume it is true for $k$, i.e.\ there exist  polynomials $f(w)$ and $g(w)$ such that $f(w)X_k(w) +g(w)Y_k(w) =1$. Let  $J$ be an ideal containing $X_{k+1}(w)$ and $Y_{k+1}(w)$. Then
\begin{eqnarray*}
(1+q^{-k}w)(1+q^k w) &=&  (f(w)X_k(w) +g(w)Y_k(w))(1+q^{-k}w)(1+q^{k}w) \\
&=&  f(w) X_{k+1}(w)(1+q^{k}w) + g(w)Y_{k+1}(w)(1+q^{-k}w) \in J.
\end{eqnarray*}
Similarly, by considering $ (f(qw)X_k(qw) +g(qw)Y_k(qw))(1+qw)(1+q^{-k+1}w)(1+q^{-k}w)$ one finds that $(1+qw)(1+q^{-k}w)(1+q^{-k+1}w)  \in J$, and hence also $(1+q^{-1}w)(1+q^{k}w)(1+q^{k-1}w)  \in J$, by symmetry. By comparing powers of $w$, one easily finds that there exist complex numbers $\alpha_0, \alpha_1, \alpha_2$ such that $(\alpha_0 +\alpha_1w)(1+q^{k}w) +\alpha_2(1+qw)(1+q^{-k+1}w) =1$. Multiplying this by $1+q^{-k}w$ we find that  $1+q^{-k}w\in J$. Again, by symmetry $1+q^{k}w\in J$, and clearly there exists a linear combination of these that makes up 1.
\end{proof}

Define $w:= {\gamma}^{-1}({1-q})qa - q $. Using relations \eqref{disc} and \eqref{calculus.d} one finds
\begin{equation}\label{omegas}
b^*db = q[N]_q \left(\frac{\gamma}{1-q}\right)^{N-1} \prod_{l=1}^{N-1} (1+q^{l}w)\, \omega_0, \quad bdb^*= [N]_q\left(\frac{ q^{-1}\gamma}{1-q}\right)^{N-1} \prod_{l=1}^{N-1} (1+q^{-l}w)\, \omega_0.
\end{equation}
Hence $\prod_{l=1}^{N-1} (1+q^{l}w)\omega_0,  \prod_{l=1}^{N-1} (1+q^{-l}w)\omega_0 \in \Omega^{1} (C^N_{q,\gamma})$ and, by Lemma~\ref{lem.calculus}, $\omega_0 \in \Omega^{1} (C^N_{q,\gamma})$. Since the calculus is closed under the $*$-operation also $\omega_0^* \in \Omega^{1} (C^N_{q,\gamma})$. Now the theorem follows.
\end{proof}

The existence of relations such as \eqref{omegas} means that the generators $\omega_0$ and $\omega_1$ of  the  left (or right) $\cO(C^N_{q,\gamma})$-module $\Omega^{(1,0)} (C^N_{q,\gamma})$ are not independent.  That $\Omega^{(1,0)} (C^N_{q,\gamma})$  is a finitely generated and projective $\cO(C^N_{q,\gamma})$-module may be argued as follows. $\Omega^{(1,0)} (C^N_{q,\gamma})$ is a submodule of the free $\cO(D_{q,\gamma})$-module  $\cO(D_{q,\gamma})\, dz$, generated by $z^{N-1}dz$ and $z^*dz$, hence it can be identified with the $\cO(C^N_{q,\gamma})$-submodule of $\cO(D_{q,\gamma})$ generated by $z^{N-1}$ and $z^*$. This is exactly the $\cO(C^N_{q,\gamma})$-module consisting of all elements in $\cO(D_{q,\gamma})$ of degree $N-1$. By Theorem~\ref{thm.free}, $\cO(D_{q,\gamma})$ is strongly graded, hence any submodule of elements of a fixed degree is  projective by \cite[A~I.3.3~Corollary]{NasVan:gra}.  Therefore  $\Omega^{(1,0)} (C^N_{q,\gamma})$ can be seen as module of sections of a non-commutative vector bundle over the quantum cone. Furthermore, since $\cO(D_{q,\gamma})$ has no invertible elements of degree $1$ or $N-1$, by the same arguments as those following the proof of Theorem~\ref{thm.free}, the module $\Omega^{(1,0)} (C^N_{q,\gamma})$ is  not free. This is in full agreement with the identification of  $\cO(C^2_{q,1-q})$ as the quantum equatorial sphere; holomorphic cotangent bundle over the sphere is not trivial. The employment of the $*$-conjugation yields similar statements for $\Omega^{(0,1)} (C^N_{q,\gamma})$.

Perhaps it is no too surprising that the quantum cones $C^N_{q,\gamma}$ admit natural complex structures, for the considerations of Section~\ref{sec.smooth} substantiate the claim that they are smooth quantum manifolds. The exactness of the volume form might indicate that quantum cones correspond to non-commutative manifolds with boundaries.  The non-triviality  of the cotangent bundle indicates that  $C^N_{q,\gamma}$ correspond to non-parallelizable manifolds. For these reasons one might think about the non-commutatrve smoothing of the cone as of blowing up the singular point into a balloon with a neck rather than as of rounding it.

\section{Outlook and speculations}
My motivation for writing this note was to draw the reader's attention to unexpected results of actions of finite groups on quantum spaces. The classically non-free actions are freed upon quantization, singularities are removed, and the geometric and topological natures of the spaces of fixed points are transformed. This opens up an exciting possibility of deploying the full power of differential geometry in areas in which it could not be utilized previously. Already the study of one of the simplest examples leads to surprises, poses new questions and opens up avenues for further research. An immediate direction of studies would be to try and understand quantum cones as Riemannian surfaces, whether from the spectral point of view \cite{Con:non} or in a more algebraic setup \cite{Maj:Rie}, \cite{BegMaj:sta}. Once these aspects are understood in sufficient depth, one can attempt to build physical models based on quantum cones. In another direction the {\em differential smoothness} of quantum cones in the sense of \cite{BrzSit:pil}, i.e.\ the existence of an isomorphism between complexes of differential and integral forms \cite{BrzElK:int}, should be investigated. The lessons thus learnt should be applied to other deformations of orbifolds including those that can be obtained by  glueing of quantum cones.

Hitherto we discussed only algebraic aspects of quantum cones, leaving aside topological issues.  In view of relations \eqref{cone.z}, the $C^*$-algebra completions of the $\cO(C^N_{q,\gamma})$ to algebras of continuous functions $\cC(C^N_{q,\gamma})$ can be carried out in the same way as for the quantum real weighted projective spaces \cite{Brz:Sei}. Consequently,  $\cC(C^N_{q,\gamma})$ is a pullback of $N$-copies of the Toeplitz algebra along the symbol maps or the algebra of continuous functions on the $N$-copies of  quantum discs glued along their circular boundaries; see \cite{CalMat:cov}. In particular  $\cC(C^2_{q,\gamma})$ is isomorphic to the algebra of continuous functions on the quantum equatorial or generic Podle\'s sphere  \cite{Pod:sph}. This identifies $C^2_{q,\gamma}$ as a deformation of a compact manifold without a boundary. Classically, volume forms for such manifolds are never exact, yet the volume forms described in Theorem~\ref{thm.complex} are exact! The exactness of the volume form is compatible with the interpretation of $C^2_{q,\gamma}$ as the cone but not as the sphere. On the other hand, the interpretation of $C^2_{q,\gamma}$ as a sphere is compatible with the non-triviality of cotangent bundles. 

The identification of $C^2_{q,\gamma}$ leads to another question: how is it possible for the sphere to arise from an action of a finite group on the disc? A possible heuristic answer might be to  think about the quantum disc as having an internal structure: the quantum disc is composed of layers of quantum discs. The $\ZZ_2$-action does not bulge  the disc into a cone, but rather it gently separates two internal layers, which are still joint at the common circle, thus forming a sphere or a balloon. Speculatively, one might interpret the classical circle that the equatorial Podle\'s sphere contains as the rim of the neck of the balloon (rather than as the equator), and thus reconcile this point of view with the interpretation of the differential smoothing offered at the end of Section~\ref{sec.complex}. For the  action of  $\ZZ_N$, $N$ internal layers are separated.

\section*{Acknowledgements.}
It is my pleasure to thank the organizers of the workshop on Noncommutative Field Theory and Gravity, Corfu 2013 for creating a friendly and inspiring research environment.

\end{document}